\documentclass{amsart}

\usepackage{mathptmx}
\usepackage{amsmath}
\usepackage{amssymb}
\usepackage{amsfonts}
\usepackage{amsxtra}

\usepackage{enumitem}
\setenumerate{label={\rm (\alph{*})}}

\usepackage[active]{srcltx}

\usepackage[color,hyperref]{paper_diening}

\usepackage{graphicx}

\newtheorem{theorem}{Theorem}[section]
\newtheorem{lemma}[theorem]{Lemma}

\newtheorem{corollary}[theorem]{Corollary}
\newtheorem{assumption}[theorem]{Assumption}

\newtheorem{definition}[theorem]{Definition}

\newtheorem{remark}[theorem]{Remark}

\numberwithin{equation}{section}

\newcommand{\meantmp}[2]{#1\langle{#2}#1\rangle}
\newcommand{\mean}[1]{\meantmp{}{#1}}

\newcommand{\ticho}[1]{}

\newcommand{\setS}{\setR^{2 \times 2}_{\sym}}

\newcommand{\pbar}{{\overline p}}
\newcommand{\qbar}{{\overline q}}
\newcommand{\RNn}{\mathbb R^{2\times 2}}
\def\seminorm{\norm}
\newcommand{\camp}{\mathcal L}
\renewcommand{\setBMO}{\rm{BMO}}

\newcommand{\Bog}{\ensuremath{\text{\rm Bog}}}

\begin{document}

\title[Campanato estimates for the generalized Stokes System]{Campanato estimates for \\the generalized Stokes System}

\author{L. Diening \and P. \Kaplicky{} \and S. Schwarzacher}

\date{Reviewed: ??}


\thanks{L. Diening and S. Schwarzacher are supported by Grant 50755385
  of DAAD. P. Kaplick\'y is supported by Grant MEB101101 of MEYS of
  the Czech Republic, and partially also by Grant GA\v CR 201/09/0917
  of the Czech Science Foundation.  }

\maketitle 

\begin{abstract}
  We study interior regularity of solutions of a generalized
  stationary Stokes problem in the plane. The main, elliptic part of
  the problem is given in the form $\divergence(\bfA(\bfD\bfu))$,
  where $\bfD$ is the symmetric part of the gradient. The model case
  is $\bfA(\bfD\bfu)=(\kappa+\abs{\bfD\bfu})^{p-2}\bfD\bfu$. We show
  optimal $\setBMO$ and Campanato estimates for $\bfA(\bfD\bfu)$. Some
  corollaries for the generalized stationary Navier-Stokes system and
  for its evolutionary variant are also mentioned.

\end{abstract}


\section{Introduction}
\label{sec:introduction}

Let $\Omega\subset\setR^2$ be a domain.  In this article we study
properties of the local weak solution $\bfu \in W^{1, \phi}(\Omega)$
and $\pi \in L^{\phi^*}(\Omega)$ of the generalized Stokes
problem\footnote{We denote by $\bfD\bfu$ the symmetric part of the
  gradient of $\bfu$, i.e. $\bfD\bfu=(\nabla\bfu+(\nabla\bfu)^T)/2$.}
\begin{align}
  \begin{alignedat}{2}
    \label{eq:sysinhom}
    -\divergence \bfA(\bfD \bfu) + \nabla \pi &=
    -\divergence \bfG &&\quad\text{in $\Omega$},
    \\
    \divergence \bfu&=0 &&\quad\text{in $\Omega$}
  \end{alignedat}
\end{align}
for given $\bfG \,:\, \Omega \to \setR^{2 \times 2}_\sym$. Here $\bfu$
stands for the velocity of a fluid and $\pi$ for its pressure.  We do
not need boundary condition, since our results are local. The model
case is $\bfA(\bfQ) = \nu(\kappa + \abs{\bfQ})^{p-2} \bfQ$
corresponding to power law fluids with $\nu > 0$, $\kappa \geq 0$,
$1<p < \infty$ and $\bfQ$~symmetric. But we also allow more general
growth conditions, which include for example Carreau type fluids
$\bfA(\bfQ) = \mu_\infty \bfQ + \nu(\kappa + \abs{\bfQ})^{p-2} \bfQ$
with $\mu_\infty\geq 0$ (see Subsection~\ref{ssec:A}).  In this
article we are interested in the qualitative properties of $\bfA(\bfD
\bfu)$ and~$\pi$ in terms of $\bfG$. The divergence form of the
right-hand side is only for convenience of the formulation of the
result, since every $\bff$ can be written as $-\divergence \bfG$ with
$\bfG$ symmetric, see Remark~\ref{rem:rhsl2}.

System \eqref{eq:sysinhom} originates in fluid mechanics. It is a
simplified stationary variant of the system
\begin{align}\label{eq:gnse-full}
  \bfu_t-\divergence\bfA(\bfD\bfu)+[\nabla\bfu]\bfu+\nabla\pi=-\divergence
  \bfG,\quad \divergence\bfu=0,
\end{align}
where $\bfu$ stands for a velocity of a fluid and $\pi$ for its
pressure. The extra stress tensor $\bfA$ determines properties of
the fluid and must be given by a constitutive law. If $\bfA(\bfQ) =
2\nu \bfQ$ with constant viscosity $\nu > 0$, then \eqref{eq:gnse-full}
is the famous Navier-Stokes system, which describes the flow of a
Newtonian fluids.  In the case of Non-Newtonian fluids however, the
viscosity is not constant but may depend non-linearly on $\bfD \bfu$.
The power law fluids and the Carreau type fluids are such examples,
which are widely used among engineers.  For a more detailed discussion
on the connection with mathematical modeling see e.g. \cite{MR2182831,
  MR1268237}. The existence theory for such type of fluids was
initiated by Ladyzhenskaya~\cite{MR0226907, 0184.52603} and
Lions~\cite{0189.40603}.


The main result of the article are the following Campanato type
estimates for the local weak solutions of~\eqref{eq:sysinhom}.
\begin{theorem}
  \label{thm:main}
  There is an $\alpha>0$ such that for all $\beta\in[0,\alpha)$ 
 a~$C>0$ exists such that for every ball $B$ with $2B \subset \Omega$
  \begin{align*}
    \seminorm{\bfA(\bfD\bfu)}_{\camp^{1,2+\beta}(B)} +
    \seminorm{\pi}_{\camp^{1,2+\beta}(B)} \leq C\,\bigg(
    \seminorm{\bfG}_{\camp^{1,2+\beta}(2B)}+
    R^{-\beta}\dashint_{2B}\abs{\bfA(\bfD\bfu)-\mean{\bfA(\bfD\bfu)}_{2B}}dx
    \bigg).
  \end{align*}
  In particular, $\bfG\in \camp^{1,2+\beta}(2B)$ implies
  $\bfA(\bfD\bfu), \pi\in \camp^{1,2+\beta}(B)$.
\end{theorem}
The spaces $\camp^{1,2+\beta}(B)$ are the Campanato spaces, see
Subsection~\ref{ssec:notation}. Our main theorem in particular
includes the $\setBMO$-case (bounded mean oscillation), since $\setBMO
= \camp^{1,2}$. Theorem~\ref{thm:main} is a consequence of the
refined $\setBMO_\omega$-estimates of Theorem~\ref{thm:BMOomega},
which also include the case $\setVMO$ (vanishing mean oscillation).
The upper bound~$\alpha$ is given by the maximal (local) regularity of
the homogeneous generalized Stokes system. Our estimates hold up to
this regularity exponent. Due to the Campanato characterization of
H{\"o}lder spaces $C^{0,\alpha}$ our results can also be expressed in
terms of H{\"o}lder spaces.

Our result is an extension of the results in~\cite{bmophi} to the
context of Non-Newtonian fluids. In~\cite{bmophi} we studied the problem
\begin{align}
  \label{eq:oldbmo}
  -\divergence(\bfA(\nabla \bfu)) &= -\divergence \bfG
\end{align}
for $\Omega \subset \setR^n$, $n \geq 1$, with similar growth
conditions on~$\bfA$ but $\bfD \bfu$ replaced by the full
gradient~$\nabla \bfu$. The model case is $\bfA(\nabla \bfu) =
\abs{\nabla \bfu}^{p-2} \nabla \bfu$ with $1<p<\infty$.
In~\cite{bmophi} we proved Theorem~\ref{thm:main} for weak
solutions of~\eqref{eq:oldbmo}. The case $p \geq 2$, has been studied
first in~\cite{DiBMan93}. 

Theorem~\ref{thm:main} is the limit case of the nonlinear
\Calderon{}-Zygmund theory, which was initiated by
\cite{Iwa1982,Iwa1983}. Iwaniec proved that $\bfG \in L^s$ with $s
\geq p'$ implies $\bfA(\nabla \bfu) \in L^s$, where $p' =
\frac{p}{p-1}$. See also~\cite{DiBMan93} for related works. In the
context of fluids the corresponding result was obtained
in~\cite{DieKap12}. However, the limiting regularity
of~\eqref{eq:sysinhom} for $\bfG=0$ restricts the transfer of
integrability to the range $s \in [p',\frac{n}{n-2}p']$ for $n \geq 3$
and $s \in [p',\infty)$ for $n=2$.

The reduced regularity for~\eqref{eq:sysinhom} with~$\bfG=0$ is
also the reason, why we can only treat the planar case $n=2$ in this
paper.  The crucial ingredient for Theorem~\ref{thm:main} are the
decay estimates for the homogeneous case~$\bfG=0$ in terms of the
gradients.  In this paper we are able to prove such decay estimates in
the planar case~$n=2$, see Theorem~\ref{thm:decay}. If such estimates
can be proven for $n\geq 3$, then Theorem~\ref{thm:main} would
directly generalize to this situation. Unfortunately, this is an open
problem, even in the absence of the pressure.

Theorem~\ref{thm:main} can be used to improve the known regularity
results for the stationary problem with convective term $[\nabla \bfu]
\bfu$, see Section~\ref{sec:NSpara}, and for the instationary
problem~\eqref{eq:gnse-full}, see Section~\ref{sec:NSpara}. The first
$C^{1,\alpha}$-regularity for planar flows were obtained in the series
of the articles~\cite{0946.76006,0978.35046,MR1727451} under various
boundary conditions under the restriction $\kappa > 0$. See also \cite{shil2000, abf2005}. The stationary
degenerate case $\kappa \geq 0$ was treated in~\cite{wolf2007} for
$1<p \leq 2$. To our knowledge the only result for $n\geq 2$ is the
one obtained in~\cite{MR2461247} with $\kappa>0$ and $1<p\leq 2$ and
small data and zero boundary values. Because of the zero boundary
values (combined with the small data), we are not able to use this
result for the higher regularity of the case~$\bfG=0$.

Note that our result is optimal with respect to the regularity
of~$\bfG$. All other planar results mentioned above need much stronger
assumptions on the regularity of~$\bfG$. This is one of the advantages
of the non-linear \Calderon{}-Zygmund theory. This is the basis for
our improved results in Section~\ref{sec:NS} and
Section~\ref{sec:NSpara} for the system including the convective term.
It is based on the fact, that the convective term can be written as
$\divergence(\bfu \otimes \bfu)$ using $\divergence \bfu=0$ and
therefore can be treated as a force term~$\divergence \bfG$.

\section{Notation, basic definitions and auxiliary results}
\label{sec:notat-basic-prop}
\par 
\subsection{Notation}
\label{ssec:notation}

By $B$ we will always denote a ball in $\setR^2$. We write $B
\compactsubset \Omega$ if the closure of $B$ is contained in $\Omega$.
Let $\abs{B}$ denote the volume of $B$.  Vector valued mappings are
denoted by bold letters, e.g. $\bfu$, while single valued functions
with regular letters, e.g. $\eta$.  For $\bff \in L^1_\loc(\setR^2)$
we define component-wise
\begin{align*}
  \mean{\bff}_B=\dashint_B \bff(x)\,dx := \frac{1}{\abs{B}} \int_B
  \bff(x)\,dx,\mbox{ and } M^\sharp_B
  \bff=\dashint_B\abs{\bff-\mean{\bff}_B}dx,\quad (M^\sharp
  \bff)(x)=\sup_{B\ni x}M^\sharp_B \bff.
\end{align*}
The space $\setBMO$ of function of bounded mean oscillations is defined via the
seminorm (for $\Omega$ open)
\begin{align*}
  \norm{\bff}_{\setBMO(\Omega)}:=\sup_{B\subset\Omega}
  \dashint_B\abs{\bff-\mean{\bff}_B}dx = \sup_{B\subset\Omega}
  M_B^\sharp \bff;
\end{align*}
saying that $\bff\in \setBMO(\Omega)$, whenever its seminorm is
bounded.  Therefore $\bff\in \setBMO(\mathbb R^2)$ if and only if
$M^\sharp \bff\in L^\infty(\mathbb R^2)$. We say that a function
$\bff\in \setBMO(\Omega)$ belongs to the subspace
$\setVMO_{loc}(\Omega)$ if $\lim_{\epsilon\to
  0+}\sup_{B\subset\Omega,\abs{B}<\epsilon}M_B^\sharp \bff=0$.  We
need also the following refinements of $\setBMO$, see~\cite{Spa65}.
For a non-decreasing function $\omega \,:\, (0,\infty) \to (0,\infty)$
we define
\begin{align*}
  M^\sharp_{\omega,B} \bff&=\frac{1}{\omega(R)}
  \dashint_B\abs{\bff-\mean{\bff}_B}dx,
\end{align*}
where $R$ is the radius of~$B$. We define the seminorm
\begin{align*}
  \norm{\bff}_{\setBMO_\omega(\Omega)}:= \sup_{B\subset\Omega}
  M_{\omega,B}^\sharp \bff.
\end{align*}
The choice $\omega(r)=1$ gives the usual $\setBMO$ seminorm, while
$\omega(r)= r^\alpha$ with $0<\alpha \leq 1$ induces the Campanato
space $\camp^{1,2+\alpha}(\Omega)$.
Its seminorm we denote $\seminorm{\bfu}_{\camp^{1,2+\beta}(\Omega)}$.

By $k\,B$, with $k>0$, we denote the ball with the same center and
$k$~times the radius. For functions $f,g$ on $\Omega$ we define
$\skp{f}{g} := \int_\Omega f(x) g(x)\,dx$. Similarly also for mappings
to $\setR^n$, $n>1$.
We write $f\sim g$ if and only if there exist
constants $c_0, c_1>0$, such that
\begin{align*}
  c_0 f &\le g \le c_1 f\,,
\end{align*}
where we always indicate on what the constants may depend.
Furthermore, we use $c$, $C$ (no index) as generic constants,
i.\,e.\ their values may change line to line but does not depend on
the important quantities.

We say that a function $\rho \,:\, [0,\infty) \to [0,\infty)$ is {\em
  almost increasing} if there is $c>0$ such that for all $0\leq s\leq
t$ the inequality $\rho(s) \leq c\, \rho(t)$ is valid. We say that
$\rho$ is {\em almost decreasing} if there is $c>0$ such that for all
$0\leq s\leq t$ the inequality $\rho(s) \geq c\, \rho(t)$ is valid. We
say that $\rho$ is {\em almost monotone} if it is almost increasing or
almost decreasing.

For a mapping $\bfu:\Omega\to\setR^2$ we define
$\bfD\bfu=(\nabla\bfu+(\nabla\bfu)^T)/2$,
$\bfW\bfu=(\nabla\bfu-(\nabla\bfu)^T)/2$ and
$([\nabla\bfu]\bfu)_j=\sum_{k=1}^2\bfu_k\partial_k\bfu_j$.  In the
parts of the article dealing with evolutionary problems we will assume
that $\bfu:\Omega\times (0,T)\to\setR^2$. In this case all operators
$\nabla$, $\bfD$, $\bfW$ and $\divergence$ are understood only with
respect to the variable $x\in\Omega$. 

For $\bfP, \bfQ\in\setR^n$ with $n\geq 1$ we define
$\bfP\cdot\bfQ=\sum_{j=1}^n\bfP_j\bfQ_j$. The symbol $\setS$ denotes
the set of symmetric $2\times2$ matrices. For a set $M\subset\setR^n$
we denote $\chi_M$ as the characteristic function of the set $M$, i.e.
$\chi(x)=1$ if $x\in M$ otherwise it is equal to zero. We write
$\setR^{\geq 0}=[0,+\infty)$.

\par 
\subsection{N-functions and the extra stress tensor $\bfA$}
\label{ssec:A}

The following definitions and results are standard in the context of
N--function (see e.\,g.~\cite{RaoRen1991}).
\begin{definition}\label{def:Nfce}
  A real function $\varphi \,:\, \setR^{\geq 0} \to \setR^{\geq 0}$ is
  said to be an N-function if it satisfies the following conditions:
  There exists the derivative $\varphi'$ of $\varphi$. This derivative
  is right continuous, non-decreasing and satisfies $\varphi'(0) = 0$
  and $\varphi'(t)>0$ for $t>0$. Especially, $\varphi$ is convex.
\end{definition}
\begin{definition}\label{def:delta2}
  We say that the N-function $\varphi$ satisfies the
  $\Delta_2$--condition, if there exists $c_1 > 0$ such that for all
  $t \geq 0$ it holds $\varphi(2t) \leq c_1\, \varphi(t)$. By
  $\Delta_2(\varphi)$ we denote the smallest constant $c_1$. For a
  family $\Phi$ of N-functions we define
  $\Delta_2(\Phi) := \sup_{\phi \in \Phi}
  \Delta_2(\phi)$.
\end{definition}

Let $\varphi$ be an N-function. We state some of its basic properties.
Since $\varphi(t) \leq \varphi(2t)$ the $\Delta_2$-condition is
equivalent to $\varphi(2t) \sim \varphi(t)$. The complementary
function $\phi^*$ is given by
\begin{align*}
  \phi^*(u) := \sup_{t \geq 0} \big(ut - \phi(t)\big).
\end{align*}
It satisfies $(\phi^*)'(t) = (\phi')^{-1}(t)$, where $(\phi')^{-1}$ is
the right-continuous inverse of~$\phi'$. Moreover, $(\phi^*)^* = \phi$.


For all $\delta>0$ there exists $c_\delta$ (only depending on
$\Delta_2({\varphi^*})$) such that for all $t, u \geq 0$ 
\begin{align}
  \label{ineq:young_classical}
  t\,u &\leq \delta\, \varphi(t) + c_\delta\, \varphi^\ast(u).
\end{align}
This inequality is called Young's inequality. For all $t\geq 0$
\begin{gather}
  \label{ineq:phiast_phi_p_pre}
  \begin{aligned}
    \frac{t}{2} \varphi'\Big(\frac{t}{2} \Big) \leq \varphi(t) \leq
    t\,\varphi'(t),
    \quad
    \varphi \bigg(\frac{\varphi^\ast(t)}{t} \bigg) \leq
    \varphi^\ast(t) \leq \varphi \bigg( \frac{2\, \varphi^\ast(t)}{t}
    \bigg).
  \end{aligned}
\end{gather}
Therefore, uniformly in $t\geq 0$ 
\begin{gather}
  \label{ineq:phiast_phi_p}
  \varphi(t) \sim \varphi'(t)\,t, \qquad 
  \varphi^\ast\big( \varphi'(t) \big) \sim \varphi(t),
\end{gather}
where the constants only depend on  $\Delta_2({\phi,\varphi^*})$.


For an N-function $\phi$ with $\Delta_2(\phi)<\infty$, we denote by
$L^\phi$ and $W^{1,\phi}$ the classical Orlicz and Sobolev-Orlicz
spaces, i.\,e.\ $\bfu \in L^\phi$ if and only if 
$\int\phi(\abs{\bfu})\,dx < \infty$ and $\bfu \in W^{1,\phi}$ if and only
if $\bfu, \nabla \bfu \in L^\phi$.  By $W^{1,\phi}_0(\Omega)$ we
denote the closure of $C^\infty_0(\Omega)$ in $W^{1,\phi}(\Omega)$.

Throughout the paper we will assume that~$\phi$ satisfies the
following assumption.
\begin{assumption}
  \label{ass:phi}
  Let $\phi$ be an N-function with $\varphi\in C^2((0,+\infty))\cap
  C^1([0,+\infty))$ such that $\varphi''$ is almost monotone on
  $(0,+\infty)$ and 
  \begin{align*}
    \varphi'(t) &\sim t\,\varphi''(t)
  \end{align*}
  uniformly in $t \geq 0$.
  The constants hidden in $\sim$ are called the {\em characteristics
    of $\phi$}.
\end{assumption}
We remark that if $\phi$ satisfies Assumption~\ref{ass:phi} below,
then $\Delta_2(\set{\phi,\phi^\ast}) < \infty$ will be automatically
satisfied, where $\Delta_2(\set{\phi,\phi^*})$ depends only on the
characteristics of $\varphi$, see for example \cite{BelDieKre11} for a
proof.  Most steps in our proof do not require that $\phi''$ is almost
monotone. It is only needed in Theorem~\ref{thm:diekap} for the
derivation of the decay estimates of Theorem~\ref{thm:decay}.

Let us now state the assumptions on~$\bfA$.
\begin{assumption}
  \label{ass:A}
  Let $\phi$ hold Assumption~\ref{ass:phi}. The vector field $\bfA
  \,:\, \setR^{2\times 2} \to \setR^{2 \times 2}$, $\bfA\in
  C^{0,1}(\setR^{2\times 2}\setminus\{0\})\cap C^0(\setR^{2\times2})$
  satisfies the non-standard $\varphi$-growth condition, i.\,e.\ there
  are $c,C>0$ such that for all $\bfP,\bfQ\in\setS$ with $\bfP\not=0$
  \begin{align}
    \label{eq:A_mon}
    \begin{aligned}
      \big(\bfA(\bfP) - \bfA(\bfQ)\big) \cdot \big(\bfP - \bfQ\big) &\geq c\,
      \varphi''(\abs{\bfP}+ \abs{\bfQ})\, \abs{\bfP - \bfQ}^2,
      \\
      \abs{\bfA(\bfP) - \bfA(\bfQ)} &\leq C\, \varphi''(\abs{\bfP}+
      \abs{\bfQ})\, \abs{\bfP - \bfQ}
    \end{aligned}
  \end{align}
  holds. 
  %
  We also require that $\bfA(\bfD)$ is symmetric for all
  $\bfD\in\setS$ and $\bfA(\bfzero)=\bfzero.$
\end{assumption}
Let us provide a few typical examples.  If $\phi$ satisfies
Assumption~\ref{ass:phi}. Then both $\bfA(\bfQ):= \varphi'(\abs{\bfQ})
\frac{\bfQ}{\abs{\bfQ}}$ and $\bfA(\bfQ):= \varphi'(\abs{\bfQ^\sym})
\frac{\bfQ^\sym}{\abs{\bfQ^\sym}}$ satisfy Assumption \ref{ass:A}.
See~\cite{DieE08} for a proof of this result. In this case,
\eqref{eq:sysinhom} is just the Euler-Lagrange equation of the local
$W^{1,\phi}_{\divergence}$-minimizer of the energy $\mathcal{J}(\bfw)
:= \int_\Omega \phi(\abs{\bfD \bfw})\,dx + \skp{\bfG}{\nabla \bfw}$.
Here $W^{1,\phi}_{\divergence}$ is the subspace of functions $\bfw \in
W^{1,\phi}$ with $\divergence \bfw=0$.  The pressure acts as a
Lagrange multiplier. This includes in particular the case of power law
and Carreau type fluids:
\begin{enumerate}
\item Power law fluids with $1<p<\infty$, $\kappa \geq 0$ and $\nu>0$
  \begin{alignat*}{3}
    \bfA(\bfQ) &= \nu (\kappa + \abs{\bfQ})^{p-2} \bfQ &&\quad\text{
      and }\quad &\varphi(t) &= \int_0^t \nu (\kappa + s)^{p-2}\,s\,ds
    \\
    \intertext{or} \bfA(\bfQ) &=
    \nu(\kappa^2+\abs{\bfQ}^2)^{\frac{p-2}{2}} \bfQ&&\quad\text{ and
    }\quad &\varphi(t) &= \int_0^t
    \nu(\kappa^2+s^2)^{\frac{p-2}{2}}\,s\,ds.
  \end{alignat*}
\item Carreau type fluids with $1<p<\infty$, $\kappa,\mu_\infty \geq
  0$ and $\nu>0$
  \begin{alignat*}{3}
    \bfA(\bfQ) &= \mu_\infty \bfQ + \nu (\kappa + \abs{\bfQ})^{p-2}
    \bfQ &&\quad\text{ and }\quad &\varphi(t) &= \int_0^t \mu_\infty s
    +\nu (\kappa + s)^{p-2}\,s\,ds.
  \end{alignat*}
\item For $1<p<\infty$, $\mu_\infty>0$, and $\nu\geq 0$
  \begin{alignat*}{3}
    \bfA(\bfQ) &= \mu_\infty \bfQ + \nu\, {\rm arcsinh}( \abs{\bfQ})
    \frac{\bfQ}{\abs{\bfQ}} &&\quad\text{ and }\quad &\varphi(t) &=
    \int_0^t \mu_\infty s +\nu\, {\rm arcsinh}(s)\, ds.
  \end{alignat*}
\end{enumerate}

We introduce the family of shifted N-functions 
$\set{\varphi_a}_{a  \geq 0}$
 by $\varphi_a'(t)/t := \varphi'(a+t)/(a+t)$. If $\phi$
satisfies Assumption~\ref{ass:phi} then $\varphi_a''(t) \sim
\varphi''(a+t)$ uniformly in $a,t \geq 0$. Moreover, $\Delta_2(\set{\phi_a,(\phi_a)^*})_{a\geq 0}<\infty$ depending only on the characteristics of $\phi$.
Let use define $\bfV\,:\, \setR^{2\times 2} \to \setR^{2\times 2}$
\begin{align}
  \label{eq:defV}
  \bfV(\bfQ ) &= \sqrt{\phi'(\abs{\bfQ }) \abs{\bfQ }} \frac{\bfQ
  }{\abs{\bfQ }}.
\end{align}
In the special case of $\bfA(\bfQ):= \varphi'(\abs{\bfQ})
\frac{\bfQ}{\abs{\bfQ}}$ the quantity $\bfV(\bfQ)$ is characterized by
\begin{align*}
  \abs{\bfV(\bfQ )}^2 &= \bfA(\bfQ ) \cdot \bfQ \qquad \text{ and }
  \qquad \frac{\bfV(\bfQ )}{\abs{\bfV(\bfQ )}} = \frac{\bfA(\bfQ
    )}{\abs{\bfA(\bfQ )}} = \frac{\bfQ }{\abs{\bfQ }}.
\end{align*}
The connection between $\bfA$, $\bfV$, and the shifted N-functions is best
reflected in the following lemma, which is a summary of Lemmas~3, 21,
and 26 of~\cite{DieE08}.
\begin{lemma}
  \label{lem:hammer}
  Let $\bfA$ satisfy Assumption~\ref{ass:A}. Then
    \begin{align*}
      \big({\bfA}(\bfP) - {\bfA}(\bfQ )\big) \cdot \big(\bfP-\bfQ 
      \big) &\sim \bigabs{ \bfV(\bfP) - \bfV(\bfQ )}^2
      \\
      &\sim\phi_{\abs{\bfQ }}\big(\abs{\bfP-\bfQ }\big)
      \\
      &\sim\big(\phi^*\big)_{\abs{\bfA(\bfQ )}}\big(\abs{\bfA(\bfP)-\bfA(\bfQ )}\big)
      \intertext{uniformly in $\bfP, \bfQ  \in \setR^{2\times 2}$.
        Moreover,}
      \bfA(\bfQ ) \cdot \bfQ  \sim \abs{\bfV(\bfQ )}^2 &\sim \phi(\abs{\bfQ }), \intertext{ and
      }
      \abs{\bfA(\bfP)-\bfA{(\bfQ )}}&\sim\big(\phi_\abs{\bfQ }\big)'\big(\abs{\bfP-\bfQ }\big),
    \end{align*}
  uniformly in $\bfP,\bfQ  \in \setR^{2\times 2}$.
\end{lemma}

As a further consequence of Assumption~\ref{ass:phi} there exists
$1 < p \leq q < \infty$ and $K_1>0$ such that
\begin{align}
  \label{eq:typeT}
  \phi(st) \leq K_1\, \max \set{s^p, s^q} \phi(t)
\end{align}
for all $s,t \geq 0$. The exponents $p$ and $q$ are called the lower
and upper index of~$\phi$, respectively.
We say that $\phi$ is of
type $T(p,q,K_1)$ if it satisfies~\eqref{eq:typeT}, where we allow
$1\leq p \leq q < \infty$ in this definition. 

The following two lemmas show an important invariance in terms of shifts.
\begin{lemma}[{Lemma~22, \cite{DieK08}}]
  \label{lem:shiftdual}
  Let $\phi$ hold Assumption \ref{ass:phi}. Then
  $(\phi_{\abs{P}})^*(t)\sim (\phi^*)_{\abs{\bfA(P)}}(t)$ holds
  uniformly in $t\geq 0$ and $P \in \RNn$. The implicit constants
  depend on $p$, $q$ and $K_1$ only.
\end{lemma}
We define
\begin{align}
  \label{eq:bar}
  \pbar := \min \set{p,2}\text{ and }\qbar := \max \set{q,2}.
\end{align} 
\begin{lemma} 
  \label{lem:phi1A} 
  Let $\phi$ be of type $T(p,q,K_1)$ and $P\in \RNn$, then there is a $K$ depending on $K_1,p,q$ such that
  $\phi_{\abs{P}}$ is of type $T(\pbar, \qbar, K$) and
  $(\phi_{\abs{P}})^*$ and $(\phi^*)_{\abs{\bfA(P)}}$ are of type
  $T(\qbar', \pbar', K)$.
\end{lemma}
\noindent
The proof can be found in \cite[Lemma~2.3]{bmophi}.  Finally, we can
deduce from Lemma~\ref{lem:phi1A} the following versions of {\em
  Young's inequality}. For all $\delta \in (0,1]$ and all $t, s \geq
0$ it holds
\begin{align}
  \label{eq:young}
  \begin{aligned}
    t s &\leq K^{\bar{q}}\, \delta^{1-\bar{q}}\, \phi(t) +
    \delta\, \phi^\ast(s),
    \\
    t s &\leq \delta\, \phi(t) + K^{\bar{p}'-1}\,
    \delta^{1-\bar{p}'}\, \phi^\ast(s).
  \end{aligned}
\end{align}

\section{Proof of the main theorem}
\label{sec:gehring}

Let $\bfu,\pi$ be the local weak solution of~\eqref{eq:sysinhom}, in
the sense that $\bfu \in W^{1,\phi}_{\divergence}(\Omega)$, $\pi \in
L^{\phi^*}(\Omega)$, and
\begin{align}
  \label{eq:sysinhom-press}
  \forall\bfxi\in W^{1,\varphi}_0(\Omega): \skp{\bfA(\bfD
    \bfu)}{\bfD \bfxi} - \skp{\pi}{\divergence \bfxi} &=
  \skp{\bfG}{\bfD \bfxi},
\end{align}
where we used that $\bfA(\bfD \bfu)$ and $\bfG$ are symmetric.  To
omit the pressure, we will use divergence free test function, i.e.
\begin{align}
  \label{eq:sysinhom-weak}
  \forall\bfxi\in W^{1,\varphi}_{0,\divergence}(\Omega):
  \skp{\bfA(\bfD \bfu)}{\bfD \bfxi} &= \skp{\bfG}{\bfD \bfxi}.
\end{align}
The method of the proof of Theorem~\ref{thm:main} was developed in
\cite{bmophi} for elliptic systems with the main part depending on
full gradient of solutions. It is based on a reverse H{\"o}lder
inequality, an approximation by the problem with zero right hand side
and a decay estimate for this approximation. These three properties
are discussed in the subsequent subsections. Note that the restriction
to the planar case and $\phi''$ almost monotone is only needed for the
decay estimate of Subsection~\ref{ssec:decay}.  The first two
subsections are valid independently of these extra assumptions.

\subsection{Reverse H{\"o}lder inequality}
In this section we show the reverse H{\"o}lder estimate for solutions of
\eqref{eq:sysinhom}. To prove the result we need a Sobolev-Poincar{\'e}
inequality in the Orlicz setting from~\cite[Lemma 7]{DieE08}.
\begin{theorem}[Sobolev-Poincar{\'e}]
  \label{thm:poincare}
  Let $\psi$ be an N-function such that $\psi$ and $\psi^*$ satisfy the
  $\Delta_2$-condition.  Then there exists $0 < \theta < 1$ and $c>0$,
  which depend only on $\Delta_2(\set{\psi,\psi^*})$ such that
  $\psi^\theta$ is almost convex\footnote{It is proportional to a
    convex function.} and the following holds. For every ball $B
  \subset \setR^n$ with radius~$R$ and every $\bfv \in W^{1,\psi}(B)$
  holds
  \begin{align}
    \label{eq:poincare} \dashint_B \psi\bigg( \frac{\abs{\bfv -
        \mean{\bfv}_B}}{R} \bigg)dx  \leq c\, \bigg( \dashint_B
    \psi^{\theta}(\abs{\nabla \bfv})dx 
    \bigg)^\frac{1}{\theta}.
  \end{align}
\end{theorem}
\begin{remark}
  It is not possible to replace full gradient on the right hand side with the symmetric one only. Consider $\bfv=(x_2,-x_1)$ on the unit ball.
\end{remark}

We also need the following version of the Korn's inequality for Orlicz
spaces, which is a minor modification of the one in~\cite[Theorem
6.13]{DieRuSchu2010}. See~\cite{BreDie11} for sharp conditions for
Korn's inequality on Orlicz spaces. 
\begin{lemma}
  \label{lem:korn-anti}
  Let $B\subset\setR^n$ be a ball. Let $\psi$ be an N-function such
  that $\psi$ and $\psi^*$ satisfy the $\Delta_2$-condition (for
  example let $\psi$ satisfy Assumption~\ref{ass:phi}).  Then for all
  $\bfv\in W^{1,\psi}(B)$ with $\mean{\bfW\bfv}_B=0$ the inequality
  \begin{align*}
    \int_B\psi(|\nabla\bfv|) \,dx\leq C\int_B\psi(|\bfD\bfv|)\,dx
  \end{align*}
  holds.  The constant $C>0$ depends only on $\Delta_2(\set{\psi,
    \psi^*})<\infty$.  
\end{lemma}
\begin{proof}
  From~\cite[Theorem 6.13]{DieRuSchu2010} we know that
  \begin{align}
    \label{eq:korn-anti1}
    \int_B\psi(|\nabla\bfv-\mean{\nabla \bfv}_B|) \,dx\leq
    C\int_B\psi(|\bfD\bfv-\mean{\bfD \bfv}_B|)\,dx.
  \end{align}
  Using $\mean{\bfW \bfv}_B=0$ we have $\nabla \bfv = (\nabla \bfv -
  \mean{\nabla \bfv}_B) + \mean{\bfD \bfv}_B$. Thus, by triangle
  inequality and~\eqref{eq:korn-anti1} we get
  \begin{align*}
    \int_B\psi(|\nabla\bfv|) \,dx\leq
    c\int_B\psi(|\bfD\bfv-\mean{\bfD \bfv}_B|)\,dx + c \int_B
    \psi(\abs{\mean{\bfD \bfv}_B}) \,dx,
  \end{align*}
  where we also used $\Delta_2(\psi)<\infty$. Now, the claim follows
  by triangle inequality and Jensen's inequality.
\end{proof}

As in~\cite{bmophi} we need a reverse H{\"o}lder estimate for the
oscillation of the gradients. Additional difficulties arise due to the
symmetric gradient and the hidden pressure (so that the test functions
must be divergence free).
\begin{lemma}
  \label{lem:VPrev}
  Let $\bfu$ be a local weak solution of \eqref{eq:sysinhom} and $B$
  be a ball satisfying $2B\subset\Omega$.  There exists $\theta \in
  (0,1)$ and $c>0$ only depending on the characteristics of~$\phi$,
  such that for all $\bfP, \bfG_0\in \setS$,
  \begin{align*}
    \begin{aligned}
      \dashint_{B} \abs{\bfV(\bfD\bfu) - \bfV(\bfP)}^2dx &\leq c\,
      \bigg(\dashint_{2B} \abs{\bfV(\bfD \bfu) -
        \bfV(\bfP)}^{2\theta}dx\bigg)^{ \frac 1\theta}
      \\
      &\quad + c\dashint_{2B} (\phi^*)_{\abs{\bfA(\bfP)}}(\abs{\bfG
        -\bfG_0})dx
    \end{aligned}
  \end{align*}
  holds.  The constant $c>0$ depends only on the characteristics of
  $\phi\in T(p,q,K)$ and the constants in Assumption~\ref{ass:A}.
\end{lemma}
\begin{proof}
  Let $\eta \in C^\infty_0(2B)$ with $\chi_{B} \leq \eta \leq
  \chi_{3B/2}$ and $\abs{\nabla \eta} \leq c/{R}$, where $R$ is the
  radius of~$B$. We define $\bfpsi = \eta^{\qbar}(\bfu-\bfz)$, where
  $\bfz$ is a linear function such that $\mean{\bfu - \bfz}_{2B} = 0$,
  $\bfD \bfz = \bfP$, and $\bfW \bfz = \mean{\bfW \bfu}_{2B}$.  We
  cannot use $\bfpsi$ as test function in the pressure free
  formulation~\eqref{eq:sysinhom-weak}, since its divergence does not
  vanish. Therefore, we correct $\bfpsi$ by help of the \Bogovskii{}
  operator~$\Bog$ from~\cite{MR82b:35135}.  In particular, $\bfw =
  \Bog(\divergence \bfpsi)$ is a special solution of the auxiliary
  problem
  \begin{alignat*}{2}
    \divergence \bfw&=\divergence \bfpsi &&\quad\mbox{in $\frac 32B$}
    \\
    \bfw&=0&&\quad\mbox{in $\partial( \frac 32B)$.}
  \end{alignat*}
  We extend $\bfw$ by zero outside of $\frac 32 B$. It has been shown
  in ~\cite[Theorem~6.6]{DieRuSchu2010} that $\nabla \bfw$ can
  estimated by $\divergence \psi$ in any suitable Orlicz spaces. In
  our case we use the following estimate in terms of
  $\phi_{\abs{\bfP}}$.
  \begin{align*}
    \dashint_{2B}\varphi_{\abs{\bfP}}(\abs{\bfD\bfw})dx\leq
    C\dashint_{2B}\varphi_{\abs{\bfP}} (\abs{\divergence \bfpsi})\,dx.
  \end{align*}
  The constant $C>0$ depends only on the characteristics of~$\phi$.

  Using $\divergence \bfu= 0$, we have
  \begin{align*}
    \divergence \bfpsi = \nabla(\eta^{\qbar})\, (\bfu - \bfz) + \eta^\qbar
    \divergence(\bfu - \bfz) = \qbar \eta^{\qbar-1} \nabla \eta\, (\bfu - \bfz) -
    \eta^\qbar \trace \bfP.
  \end{align*}
  This implies
  \begin{align}
    \label{est:Bog}
    \dashint_{2B}\varphi_{\abs{\bfP}}(\abs{\bfD\bfw})dx &\leq
    C\dashint_{2B}\varphi_{\abs{\bfP}} \bigg(\frac{\abs{\bfu-\bfz}}{R}
    \bigg)\,dx + C\,\dashint_{2B}\varphi_{\abs{\bfP}} (\abs{\trace
      \bfP})\,dx.
  \end{align}

  We define $\bfxi := \psi - \bfw = \eta^\qbar(\bfu-\bfz)-\bfw$, then
  $\divergence \bfxi=0$, which ensures that $\bfxi$ is a valid test
  function for~\eqref{eq:sysinhom}. We get
  \begin{align}
    \label{eq:APG}
    \begin{aligned}
      \skp{\bfA(\bfD\bfu) - \bfA(\bfP)}{\eta^\qbar(\bfD\bfu - \bfP)} &=
      \skp{\bfG-\bfG_0}{\eta^\qbar(\bfD\bfu - \bfP)} +
      \skp{\bfG-\bfG_0}{(\bfu-\bfz) \otimes_{sym} \nabla (\eta^\qbar)}
      \\
      &\quad - \skp{\bfA(\bfD \bfu) - \bfA(\bfP)}{(\bfu-\bfz)
        \otimes_{sym} \nabla (\eta^\qbar) }
      \\
      &\quad -\skp{\bfG-\bfG_0}{\bfD\bfw}+\skp{\bfA(\bfD\bfu) -
        \bfA(\bfP)}{\bfD\bfw}.
    \end{aligned}
  \end{align}
  The symbol $\otimes_{sym}$ denotes the symmetric part of $\otimes$,
  i.e.$(\bff\otimes_{sym}\bfg)_{ij}:=(\bff_i\bfg_j+\bff_j\bfg_i)/2$
  for $\bff,\bfg\in\setR^2$. We divide~\eqref{eq:APG} by~$\abs{2B}$
  and estimate the two sides.  Concerning the left hand side we find
  by Lemma~\ref{lem:hammer}
  \begin{align*}
    \abs{2B}^{-1}\skp{\bfA(\bfD\bfu) - \bfA(\bfP)}{\eta^\qbar(\bfD\bfu -
      \bfP)} \sim \dashint_{2B} \eta^\qbar\abs{\bfV(\bfD\bfu) -
      \bfV(\bfP)}^2dx =:(I).
  \end{align*}
  We estimate the right hand side of~\eqref{eq:APG} by Young's
  inequality \eqref{eq:young} for $\phi_{\abs{\bfP}}$ with $\delta \in
  (0,1)$ using also
  $(\phi_{\abs{\bfP}})^* \sim (\phi^*)_{\abs{\bfA(\bfP)}}$ (see
  Lemma~\ref{lem:shiftdual}).
  \begin{align*}
    \begin{aligned}
      (I)\leq &
      c_\delta\dashint_{2B}(\varphi^*)_{\abs{\bfA(\bfP)}}(\abs{\bfG-\bfG_0})dx
      +\delta\dashint_{2B}
      \eta^{\pbar\qbar}\varphi_{\abs{\bfP}}(\abs{\bfD\bfu-\bfP})dx
      \\
      &+c_\delta\dashint_{2B}\varphi_{\abs{\bfP}}
      \bigg(\frac{\abs{\bfu-\bfz}}{R}\bigg)dx
      +c_\delta\dashint_{2B}\varphi_{\abs{\bfP}}(\abs{\bfD\bfw})dx
      \\
      &+ \delta \dashint_{2B} \eta^{(\qbar-1)\overline{q}'}
      (\varphi^*)_{\abs{\bfA(\bfP)}}( \abs{\bfA(\bfD\bfu)-\bfA(\bfP)}
      \big)\, dx
      \\
      &=: (II)+(III)+(IV)+(V) + (VI).
    \end{aligned}
  \end{align*}
  Now we use Lemma~\ref{lem:hammer} to estimate $(III) + (VI) \leq
  \delta\,c (I)$, so these terms can be absorbed. Moreover,
  by~\eqref{est:Bog}
  \begin{align*}
    (IV) + (V) &\leq c\, (IV) + c\,\dashint_{2B}\varphi_{\abs{\bfP}}
    (\abs{\trace \bfP})\,dx.
  \end{align*}
  Since $\bfP$ is constant, $\trace \bfP=\divergence \bfz$ and
  $\divergence \bfu=0$, we can estimate
  \begin{align}
    \label{eq:Bog_P}
    \dashint_{2B}\varphi_{\abs{\bfP}} (\abs{\trace \bfP})\,dx =
    \bigg(\dashint_{2B} (\varphi_{\abs{\bfP}})^\theta
    (\abs{\divergence(\bfu - \bfz)})\,dx \bigg)^{\frac 1 \theta} \leq
    \bigg(\dashint_{2B} (\varphi_{\abs{\bfP}})^\theta (\abs{\bfD\bfu -
      \bfD\bfz)})\,dx \bigg)^{\frac 1 \theta}.
  \end{align}
  It remains to estimate (IV). We use Sobolev-\Poincare{} inequality
  of Theorem~\ref{thm:poincare} with $\psi = \phi_{\abs{\bfP}}$ such
  that $(\phi_\abs{\bfP})^\theta$ is almost convex and
  \begin{align*}
    (IV) =
    c\,\dashint_{2B}\varphi_{\abs{\bfP}}\bigg(\frac{\abs{\bfu-\bfz}}{R}\bigg)dx
    &\leq
    c\,\bigg(\dashint_{2B}\varphi_{\abs{\bfP}}^\theta({\abs{\nabla\bfu-\nabla\bfz}})dx
    \bigg)^{\frac{1}{\theta}}
  \end{align*}
  with $\theta\in(0,1)$. The constants and~$\theta$ are independent
  of~$\abs{\bfP}$, since the $\Delta_2(\set{\phi_a}_{a\geq 0})$ is
  bounded in terms of the characteristics of $\phi$.

  As $\mean{\bfW(\bfu-\bfz)}_{2B}=0$ we find by Korn's inequality
  (Lemma~\ref{lem:korn-anti}) with $\psi = \phi_{\abs{\bfP}}^\theta$
  (almost convex) and $\bfD\bfz=\bfP$ that
  \begin{align*}
    (IV) \leq
    c\,\bigg(\dashint_{2B}\varphi_{\abs{\bfP}}^\theta(\abs{\bfD\bfu-
      \bfD\bfz})dx \bigg)^{\frac{1}{\theta}}.
  \end{align*}
  The above estimates and Lemma~\ref{lem:hammer} show that
  \begin{align*}
    (IV)+(V) &\leq
    c\,\bigg(\dashint_{2B}\varphi_{\abs{\bfP}}^\theta(\abs{\bfD\bfu-
      \bfD\bfz})dx \bigg)^{\frac{1}{\theta}} \leq
    c\,\bigg(\dashint_{2B}\abs{\bfV(\bfD \bfu) -
      \bfV(\bfP)}^{2\theta}dx\, \bigg)^{ \frac 1\theta}.
  \end{align*}
  The lemma is proved.
\end{proof}
Lemma~\ref{lem:VPrev} allows to obtain exactly as in~\cite{bmophi} the
next corollary, compare with \cite[Corollary~3.5]{bmophi}.
\begin{corollary}
  \label{cor:VPL1}
  Let the assumptions of Lemma~\ref{lem:VPrev} be satisfied. Then for
  all $\bfP \in \setS$
  \begin{align*}
    \begin{aligned}
      \dashint_{B} \abs{\bfV(\bfD\bfu) - \bfV(\bfP)}^2 dx&\leq c\,
      (\varphi^*)_{\abs{\bfA(\bfP)}}\bigg(\dashint_{2B} \abs{\bfA(\bfD
        \bfu) - \bfA(\bfP)}\,dx\bigg)
      \\
      &\quad + c(\phi^*)_{\abs{\bfA(\bfP)}}(\norm{\bfG}_{BMO(2B)})
    \end{aligned}
  \end{align*}
  The constants only depend on the characteristics of $\phi$ and the
  constants in Assumption~\ref{ass:A}.
\end{corollary}

\subsection{Approximation property}

Let $\bfu$ be a local weak solution of \eqref{eq:sysinhom} and $B$ be
a ball satisfying $2B\subset\Omega$. We consider a solution $\bfh,\rho$ of
the homogeneous problem
\begin{align}
  \begin{alignedat}{2}
    \label{eq:syshom}
    -\divergence \bfA(\bfD \bfh)+ \nabla \rho &= 0
    &&\quad\text{in $\Omega$},
    \\
    \divergence \bfu&=0 &&\quad\text{in $\Omega$},
    \\
    \bfh&=\bfu &&\quad\text{on $\partial\Omega$}.
  \end{alignedat}
\end{align}
The next lemma estimates the natural distance between $\bfu$ and its
approximation $\bfh$.
\begin{lemma}
  \label{lem:choice}
  For every $\delta>0$ there exists $c_\delta \geq 1$ such that
  \begin{align*}
    \dashint_B \abs{\bfV(\bfD \bfu) - \bfV(\bfD \bfh)}^2\,dx \leq
    \delta\,
    &(\varphi^*)_{\abs{\mean{\bfA(\bfD\bfu)}_{2B}}}\bigg(\dashint_{2B}
    \abs{\bfA(\bfD \bfu) - \mean{\bfA(\bfD\bfu)}_{2B}}dx\bigg)
    \\
    \quad + c_\delta
    &(\phi^*)_{\abs{\mean{\bfA(\bfD\bfu)}_{2B}}}(\norm{\bfG}_{BMO(2B)})
  \end{align*}
  holds. The constants depend only on the characteristics of $\varphi$
  and the constants in Assumption~\ref{ass:A}.
\end{lemma}
\begin{proof}
  The estimate is obtained by testing the difference of the equations
  for $\bfu$ and $\bfh$ by $\bfu-\bfh.$ The proof is exactly as
  in~\cite[Lemma~4.2]{bmophi}. One just needs to replace the
  gradient by the symmetric gradient.
\end{proof}

\subsection{Decay estimate}
\label{ssec:decay}

In this section we derive decay estimates for our approximation~$\bfh$.
The main ingredient is the following theorem which can be found in
\cite[Theorem~3.6]{DieKap12}. It is valid in any dimension but needs
$\phi''$ to be almost monotone. This is the only place in the paper,
where we need this assumption on~$\phi''$. 

\begin{theorem}
  \label{thm:diekap}
  Let $\phi''$ be almost monotone. If $\bfh$ is a weak solution
  of~\eqref{eq:syshom}, then there is an $r>2$ such that for every
  ball $Q\subset B$ with radius $R>0$
  \begin{align*}
    R^2\bigg(\dashint_{\frac12 Q} \abs{\nabla\bfV(\bfD\bfh)}^rdx\bigg)^{\frac2r}
    \leq 
    &\,C \dashint_{Q} \abs{\bfV(\bfD\bfh)  - \mean{\bfV(\bfD
        \bfh)}_{Q}}^2dx.
  \end{align*}
  The constants $C$ and $r$ depend only on the characteristics of
  $\varphi$ and the constants in Assumption~\ref{ass:A}.
\end{theorem}
The regularity $\bfV \in W^{1,r}$ with $r>2$ ensures in two space
dimensions that $\bfV$ is H{\"o}lder continuous. This is the reason, why our estimates can only be applied to planar flows. It is an open
question if $\bfV(\nabla \bfu)$ is H{\"o}lder continuous in higher
dimensions. 

This provides the following decay estimates in the plane:
\begin{theorem}
  \label{thm:decay}
  There exists $\gamma>0$ such that for every $\lambda \in
  (0,1]$
  \begin{align*}
    \dashint_{\lambda B} \abs{\bfV(\bfD\bfh)  - \mean{\bfV(\bfD
        \bfh)}_{\lambda B}}^2dx&
    \leq {C}{\lambda^{2\gamma}}
    \dashint_{B} \abs{\bfV(\bfD\bfh)  - \mean{\bfV(\bfD
        \bfh)}_{B}}^2dx.
  \end{align*}
  The constant $C$ and $\gamma$ depend only on the characteristics of
  $\varphi$ and the constants in Assumption~\ref{ass:A}.
\end{theorem}
\begin{proof}
  The result is clear if $\lambda \geq \frac12$, so we can assume
  $\lambda\in(0,\frac 12)$.  Let $R$ denote the radius of~$B$.  We
  compute by \Poincare inequality on~$\lambda B$, Jensen's inequality
  with $r>2$, enlarging the domain of integration and
  Theorem~\ref{thm:diekap}
  \begin{align*}
    \dashint_{\lambda B} &\abs{\bfV(\bfD\bfh) - \mean{\bfV(\bfD
        \bfh)}_{\lambda B}}^2 dx\leq C (\lambda R)^2 \dashint_{\lambda
      B} \abs{\nabla\bfV(\bfD\bfh)}^2 dx
    \\
    \leq &\,C (\lambda R)^2 \bigg( \dashint_{\lambda B}
    \abs{\nabla\bfV(\bfD\bfh)}^rdx\bigg)^\frac 2r \leq C
    R^2\lambda^{2(1-\frac 2r)} \bigg(\dashint_{\frac12 B}
    \abs{\nabla\bfV(\bfD\bfh)}^rdx\bigg)^{\frac2r}
    \\
    \leq &\,C \lambda^{2(1-\frac 2r)} \dashint_{B} \abs{\bfV(\bfD\bfh)
      - \mean{\bfV(\bfD \bfh)}_{B}}^2dx.
  \end{align*}
  As $r>2$ the proof is completed.
\end{proof}

\subsection{Proof of the main theorem}\label{mt}

Theorem~\ref{thm:main} is a corollary of the following more general
theorem.
\begin{theorem}
  \label{thm:BMOomega}
  Let $B \subset \setR$ be a ball. Let $\bfu$, $\pi$ be a local weak
  solution of \eqref{eq:sysinhom} on $2B$, with $\phi$ and $\bfA$
  satisfying Assumption~\ref{ass:A}.  Let $\omega \,:\, (0,\infty) \to
  (0,+\infty)$ be non-decreasing such that for some $\beta \in
  (0,\frac{2\gamma}{\pbar'})$ the function $\omega(r) r^{-\beta}$ is
  almost decreasing, where $\gamma$ is defined in
  Theorem~\ref{thm:decay} and $\pbar$ in~\eqref{eq:bar}.  Then
  \begin{align}\label{est:bmo}
    \norm{\pi}_{\setBMO_\omega(B)}
    +\norm{\bfA(\bfD\bfu)}_{\setBMO_\omega(B)} &\leq c\,
    M_{\omega,2B}^\sharp (\bfA(\bfD\bfu)) + c
    \norm{\bfG}_{\setBMO_\omega(2B)}.
  \end{align}
  The constants depend only on the characteristics of $\varphi$ and
  the constants in Assumption~\ref{ass:A}.
\end{theorem}
\begin{proof}
  The proof of the estimate of $\bfA(\bfD\bfu)$ follows line by line
  the proof of \cite[Theorem~5.3]{bmophi}. It is based on
  Corollary~\ref{cor:VPL1}, Lemma~\ref{lem:choice} and
  Theorem~\ref{thm:decay}. We will therefore omit the proof here and
  restrict ourselves to the additional estimates for the pressure.
 
  We define $\bfH=\bfA(\bfD\bfu)-\bfG$. It holds
  $\bfH\in\setBMO_\omega(B)\subset\setBMO(B)$.  We fix a ball
  $Q\subset B$. Then equation \eqref{eq:sysinhom} implies that
  \begin{align}\label{eq:press}
    \forall\bfxi\in W^{1,2}_0(\Omega):
    \skp{\pi-\mean{\pi}_Q}{\divergence \bfxi} =
    \skp{\bfH-\mean{\bfH}_Q}{\nabla \bfxi}.
  \end{align}
  Let $\bfxi \in W^{1,2}_0(Q)$ be the solution of the auxiliary
  problem
  \begin{align*}
    \divergence \bfxi=\pi-\mean{\pi}_Q\quad\mbox{in $Q$,}\qquad
    \bfxi=0\quad\mbox{on $\partial Q$.}
  \end{align*}
  The existence of such a solution is ensured by the \Bogovskii{}
  operator~\cite{MR82m:26014} and we have
  $\norm{\nabla\bfxi}_{L^2(Q)}\leq C\norm{\pi-\mean{\pi}_Q}_{L^2(Q)}$.
  The constant $C>0$ is independent of $Q$. Inserting such $\bfxi$
  into \eqref{eq:press} we get
  \begin{align*}
    \norm{\pi -\mean{\pi}_Q}_{L^2(Q)}^2 =
    \skp{\pi-\mean{\pi}_Q}{\divergence \xi} = \skp{\bfH -
      \mean{\bfH}_Q}{\nabla \xi}.
  \end{align*}
  This and $\norm{\nabla\bfxi}_{L^2(Q)}\leq
  C\norm{\pi-\mean{\pi}_Q}_{L^2(Q)}$ implies $ \norm{\pi
    -\mean{\pi}_Q}_{L^2(Q)} \leq c\, \norm{\bfH
    -\mean{\bfH}_Q}_{L^2(Q)}$.  We find by Jensen's inequality
  \begin{align*}
    \big(M^\sharp_{Q}\pi\big)^2\leq\dashint_Q\abs{\pi-\mean{\pi}_Q}^2dx\leq
    C \dashint_Q\abs{\bfH-\mean{\bfH}_Q}^2dx \leq 
    C\norm{\bfH}_{\setBMO(Q)}^2.
  \end{align*}
  In the last inequality we used the John-Nirenberg estimate,
  \cite[Corollary~6.12]{duoan}. It follows that $\pi\in\setBMO(B)$ and
  $\norm{\pi}_{\setBMO(Q)}\leq C\norm{\bfH}_{\setBMO(Q)}$. This
  implies that
  \begin{align*}
    M^\sharp_{\omega,Q}(\pi) \leq
    C\frac1{\omega(R_Q)}\norm{\bfH}_{\setBMO(Q)}\leq
    C\norm{\bfH}_{\setBMO_\omega(B)}
  \end{align*}
  using the monotonicity of~$\omega$. Since $Q$ is arbitrary, we have
  $\norm{\pi}_{\setBMO_\omega(B)} \leq
  \norm{\bfH}_{\setBMO_\omega(B)}$. Now $\bfH = \bfA(\bfD \bfu) -
  \bfG$ and the estimate for $\bfA(\bfD \bfu)$ conclude the proof.
\end{proof}
The choice $\omega(t)=1$ in Theorem~\ref{thm:BMOomega} gives the
$\setBMO$ estimate. However, the choice $\omega(t)=t^\beta$,
$\beta\in(0,2\gamma/\pbar')$ Theorem~\ref{thm:BMOomega} gives the
estimates in Campanato space $\camp^{1,2+\beta}$, compare
\cite[Corollary~5.5]{bmophi}. This is just Theorem~\ref{thm:main}.
\begin{remark}
  \label{rem:du-hol}
  It is possible to transfer the H{\"o}lder continuity of $\bfA(\bfD
  \bfu)$ to $\bfD \bfu$ and $\nabla \bfu$. Let us discuss the case of
  power-law and Carreau type fluids. This follows from the fact that
  $\bfA^{-1} \in C_{\loc}^{0,\sigma}$ for some $\sigma>0$. If $\kappa
  =0$, then $\sigma = \min \set{1, p'-1}$.  If $\kappa > 0$, then
  $\sigma=1$.  Now, $\bfA(\bfD \bfu) \in C^{0,\beta}$ implies $\bfD
  \bfu \in C^{0,\beta \sigma}$. Due to Korn's inequality we get
  $\nabla \bfu \in C^{0,\beta \sigma}$ as well.

\end{remark}

\begin{remark} 
  Note that if $\bfG\in \setVMO(2B)$ in Theorem~\ref{thm:BMOomega} we
  get that $\bfA(\bfD\bfu)\in\setVMO(B)$. Indeed, since $\bfG\in
  \setVMO(2B)$ there exists a nondecreasing function
  $\tilde{\omega}:(0,\infty)\to (0,\infty)$ with $\lim_{r \to 0}
  \tilde{\omega}(r)= 0$, such that $\norm{\bfG}_{\setBMO(B_r)}\leq
  \tilde{\omega}(r)$, for all $B_r\subset 2B$. Defining
  $\omega(r)=\min\{\tilde{\omega}(r),r^\frac{\alpha}{\pbar'}\}$ we
  obtain by Theorem~\ref{thm:BMOomega} the $\setBMO_\omega$-estimate
  for $\bfA(\bfD \bfu)$ and $\pi$, which imply that both are in
  $\setVMO$ (compare \cite[Corollary~5.4]{bmophi}).
\end{remark}
\begin{remark}\label{rem:rhsl2}
  Let us now assume that the right hand side of~\eqref{eq:sysinhom}
  is not given in divergence form $-\divergence \bfG$ with $\bfG$
  symmetric, but rather as~$\bff \in L^s$ with $s \geq 2$.

  Let $\bfw\in W^{2,s}(2B)\cap W^{1,s}_0(2B)$ and $\sigma\in
  W^{1,s}(2B)$ with $\mean{\sigma}_{2B}=0$ be the unique solution of
  the Stokes problem $-\divergence\bfD\bfw+\nabla\sigma=\bff$ and
  $\divergence\bfw=0$ in $2B$ with $\bfw =0$ on $\partial(2B)$. Then
  $\bff=-\divergence \bfG$ for $\bfG := \bfD\bfw - \sigma \identity$
  and $\bfG$ is symmetric.  If $s =2$, then $\bfG \in W^{1,2}(2B)
  \embedding \setVMO(2B)$. If $s >2$, then $\bfG \in W^{1,s}(2B)
  \embedding \mathcal{L}^{1,2+(1-\frac 2s)}(2B) = C^{0,1-\frac
    2s}(2B)$. In particular, Theorem~\ref{thm:BMOomega} is applicable
  and for all $s \geq 2$
  \begin{align*}
    \norm{\pi}_{\mathcal{L}^{1,2+\beta}(B)} +
    \norm{\bfA(\bfD\bfu)}_{\mathcal{L}^{1,2+\beta}(B)} \leq c\,
    R^{-\beta} M_{2B}^\sharp (\bfA(\bfD\bfu)) + c
    \norm{\bff}_{L^s(2B)}
  \end{align*}
  for $s \geq 2$ and $\beta \in (0,1-\frac 2s] \cap
  (0,\frac{2\gamma}{\pbar'})$. We additionally get $\setVMO$ estimates
  if $s=2$.

  The case $s=2$ is obviously the limiting one in this setting. In the
  case of the $p$-Laplacian, i.e. no symmetric gradient and no
  pressure, it has been proven in~\cite{CiaMaz11,DuzMin10Lip} that
  $\bff \in L^{n,1}(\setR^n)$ (Lorentz space; subspace of $L^n$)
  implies $\bfA(\nabla \bfu) \in L^\infty$.  It is an interesting open
  problem, if this also holds for the system with pressure and
  symmetric gradients (at least in the plane). Note that our results
  imply in this situation $\bfA(\bfD \bfu), \pi \in \setVMO$ for $n=2$.


\end{remark}

\section{An application to the stationary Navier-Stokes problem}
\label{sec:NS}

In this section we present an application of the previous results to
the generalized Navier-Stokes problem. We assume that $\bfu \in W^{1,
  \phi}(\Omega)$, $\divergence \bfu=0$ and $\pi \in
L^{\phi^*}(\Omega)$ are local weak solutions of the generalized
Navier-Stokes problem, in the sense that
\begin{align}
  \label{eq:sysinhomns}
  \forall\bfxi\in W^{1,\varphi}_0(\Omega):
  \skp{\bfA(\bfD \bfu)}{\bfD \bfxi} - \skp{\pi}{\divergence \bfxi} &=
  \skp{\bfG+\bfu\otimes\bfu}{\bfD \bfxi}
\end{align}
for a given mapping $\bfG: \Omega\to\setS.$

In order to handle the convective term we need the condition
\begin{align}\label{ass:conv1}
  \liminf_{s\to+\infty}\frac{\varphi(s)}{s^r}>0\quad\mbox{for some $r>\frac32$}.
\end{align}

We have the following result
\begin{theorem}
  \label{thm:genNS}
  Let $\phi$ and $\bfA$ satisfy Assumption~\ref{ass:A}
  and~\eqref{ass:conv1}. Let $\bfu$ be a local weak solution of
  \eqref{eq:sysinhomns} on $\Omega$.  Let $\beta \in
  (0,\frac{2\alpha}{\pbar'})$ ($\alpha$ is defined in
  Theorem~\ref{thm:decay} and $\pbar$ in Lemma~\ref{lem:phi1A}).  If
  $B$ is a ball with $2B\subset\Omega$ and $\bfG\in
  \camp^{1,2+\beta}(2B)$, then $\bfA(\bfD\bfu),\pi\in
  \camp^{1,2+\beta}(B)$.
\end{theorem}
\begin{proof}
  According to \cite[Remark~5.3]{DieKap12} we get that $\bfD\bfu\in L^q(3B/2)$ for all $q>1$. Consequently by the Korn inequality and the  Sobolev embedding we get that $\bfu\otimes \bfu\in \camp^{1,n+\beta}(3B/2)$. Applying Theorem~\ref{thm:main} we get the result.
\end{proof}
Exactly as in Remark~\ref{rem:du-hol} it is possible to transfer the
H{\"o}lder continuity of~$\bfA(\bfD \bfu)$ to $\bfD \bfu$ and $\nabla \bfu$.

\begin{remark}
  A similar result has been proved also in~\cite{0978.35046}, provided
  $\kappa>0$, by a completely different method, which requires the
  stronger assumption $\divergence\bfG\in L^q(2B)$ for some $q>2$.
  
  The same result was also proved in \cite{wolf2007} for power law
  fluids with $p\in(3/2,2]$ and $\kappa\geq 0$, again under the
  stronger assumption $\divergence\bfG\in L^q(2B)$ for some $q>2$.
  
  By our method we reprove these known results and improve them by
  weakening the assumption on the data of the problem.
\end{remark}

\section{An Application to the parabolic Stokes problem}
\label{sec:NSpara}

Now we apply the previous results to the evolutionary variant of the
problem \eqref{eq:sysinhom}. We set $T>0$ and $I=(0,T)$,
$\Omega_T=\Omega\times I$ and assume that $\bfu\in
L^\infty(I,L^2(\Omega))$ with $\bfD\bfu\in L^\varphi(\Omega_T)$ is a
local weak solution of the problem
\begin{align}
  \begin{alignedat}{2}
    \label{eq:sysinhom-ev}
    \partial_t \bfu - \divergence(\bfA(\bfD \bfu)) + \nabla \pi&= \bff
    &&\quad\text{in $\Omega_T$,}
    \\
    \divergence \bfu&=0 &&\quad\text{in $\Omega_T$.}
  \end{alignedat}
\end{align}

If the system of equations \eqref{eq:sysinhom-ev} is complemented by a suitable boundary and initial condition and if the data of the problem are sufficiently smooth it is possible to show existence of  a solution that moreover satisfies 
\begin{align}\label{reg:time-der}
  \partial_t\bfu\in L^\infty(I,L^2(\Omega)),
\end{align}
see for example \cite{KMS-Nodea, K-ZAA, BEKP-MMAS}. If we know such regularity of $\partial_t\bfu$ and $\bff$ is smooth, it is easy to reconstruct the pressure $\pi$ in such a way that $\pi\in L^q(\Omega_T)$ with some $q>1$ and 
\begin{align}
  \label{eq:sysinhom-ev2}
  \forall\bfxi\in C^\infty_0(\Omega_T): \int_0^T
  \!\!-\skp{\partial_t\bfu}{\bfxi}+\skp{\bfA(\bfD
    \bfu)-\pi\bfI}{\nabla \bfxi \,}dt = \int_0^T\skp{\bff}{\bfxi}
  \,dt.
\end{align}
The constant $q$ is determined by the requirement $\bfA(\bfD\bfu)\in L^q(\Omega_T)$. 

Applying the results from the previous
sections of this article we obtain the next simple corollary.
\begin{corollary}\label{cor:51}
  Let $\bfA$ and $\varphi$ satisfy Assumption~\ref{ass:A}.
  Let $\bfu \in L^\infty(I,L^2(\Omega))$ with $\bfD\bfu\in
  L^\varphi(\Omega_T)$ and $\divergence \bfu=0$ in $\Omega_T$ solve
  the problem \eqref{eq:sysinhom-ev} and satisfy \eqref{reg:time-der}.
  Let $B$ be a ball with $2B\subset\Omega$ and
  $\bff\in L^\infty(I, L^2(\Omega))$. Then
  $\bfA(\bfD\bfu),\pi\in L^\infty(I,\setVMO(B))$.
\end{corollary} 
\begin{proof}
  The result is immediate consequence of $\partial_t \bfu \in
  L^\infty(I, L^2(\Omega))$ and Remark~\ref{rem:rhsl2}.
\end{proof}
\begin{remark}
  Certainly, we can obtain a similar result for the
  problem~\eqref{eq:sysinhom-ev} with convection, as soon as $\bfu
  \otimes \bfu \in L^\infty(I, \setVMO(\Omega))$. This follows for
  example from the fact that $\bfV(\bfD \bfu) \in W^{1,2}(I,
  L^2(\Omega)) \cap L^2(I, W^{1,2}(\Omega))$. Such kind of regularity
  is obtained, if it is possible to test with $\partial_t^2 \bfu$ and
  $\Delta \bfu$.
\end{remark} 

In \cite{KMS-Nodea} a method was developed to construct
regular solutions of \eqref{eq:sysinhom-ev}. The essential assumption
was that the growth of $\bfA$ is sufficiently fast. It was necessary
to assume that
\begin{align}\label{ass:conv2}
  \liminf_{s\to+\infty}\frac{\varphi(s)}{s^r}>0\quad\mbox{for some
    $r>\frac43$}.
\end{align}
This assumption was not due to the presence of the convective term in
the analysis of \cite{KMS-Nodea}. It was necessary to overcome
problems connected with the anisotropy of the evolutionary problem
\eqref{eq:sysinhom-ev}.  The previous corollary is a first step to
improve these results. If it is possible to show $\partial_t \bfu \in
L^\infty(I, L^s(\Omega))$ for some $s>2$. Then for $\bff \in
L^\infty(I, L^s(\Omega))$, we find by Remark~\ref{rem:rhsl2} that
$\bfA(\bfD \bfu) \in L^\infty(I, C^{0,\beta}(\Omega))$ for $\beta \in
(0,1-\frac 2s] \cap (0,\frac{2\gamma}{\pbar'})$. This implies
(locally) bounded gradients~$\nabla \bfu$. So far the results of this
paper are of local nature. An extension of this technique up to the
boundary would imply globally bounded gradients~$\nabla \bfu$ and we
could reconstruct the result of \cite{KMS-Nodea} for the generalized
Stokes problem without the restriction \eqref{ass:conv2}.

\bibliographystyle{amsplain}

\def\ocirc#1{\ifmmode\setbox0=\hbox{$#1$}\dimen0=\ht0 \advance\dimen0
  by1pt\rlap{\hbox to\wd0{\hss\raise\dimen0
  \hbox{\hskip.2em$\scriptscriptstyle\circ$}\hss}}#1\else {\accent"17 #1}\fi}
  \def\polhk#1{\setbox0=\hbox{#1}{\ooalign{\hidewidth
  \lower1.5ex\hbox{`}\hidewidth\crcr\unhbox0}}} \def\cprime{$'$}
  \def\ocirc#1{\ifmmode\setbox0=\hbox{$#1$}\dimen0=\ht0 \advance\dimen0
  by1pt\rlap{\hbox to\wd0{\hss\raise\dimen0
  \hbox{\hskip.2em$\scriptscriptstyle\circ$}\hss}}#1\else {\accent"17 #1}\fi}
  \def\ocirc#1{\ifmmode\setbox0=\hbox{$#1$}\dimen0=\ht0 \advance\dimen0
  by1pt\rlap{\hbox to\wd0{\hss\raise\dimen0
  \hbox{\hskip.2em$\scriptscriptstyle\circ$}\hss}}#1\else {\accent"17 #1}\fi}
  \def\ocirc#1{\ifmmode\setbox0=\hbox{$#1$}\dimen0=\ht0 \advance\dimen0
  by1pt\rlap{\hbox to\wd0{\hss\raise\dimen0
  \hbox{\hskip.2em$\scriptscriptstyle\circ$}\hss}}#1\else {\accent"17 #1}\fi}
  \def\ocirc#1{\ifmmode\setbox0=\hbox{$#1$}\dimen0=\ht0 \advance\dimen0
  by1pt\rlap{\hbox to\wd0{\hss\raise\dimen0
  \hbox{\hskip.2em$\scriptscriptstyle\circ$}\hss}}#1\else {\accent"17 #1}\fi}
  \def\cprime{$'$}
\providecommand{\bysame}{\leavevmode\hbox to3em{\hrulefill}\thinspace}
\providecommand{\MR}{\relax\ifhmode\unskip\space\fi MR }
\providecommand{\MRhref}[2]{%
  \href{http://www.ams.org/mathscinet-getitem?mr=#1}{#2}
}
\providecommand{\href}[2]{#2}


\end{document}